\documentclass[11pt,a4paper]{article}
\usepackage{a4wide}
\usepackage[utf8]{inputenc}
\usepackage{latexsym}
\usepackage{amsmath}
\usepackage{amsthm}
\usepackage{amssymb,enumerate}
\usepackage[usenames]{color}
\usepackage{graphicx}
\usepackage{fancybox}
\usepackage[normalem]{ulem}
\usepackage{hyperref}
\usepackage{pgfplots}
\pgfplotsset{compat=1.15}
\theoremstyle{plain}
\newtheorem{proposition}{Proposition}[section]
\newtheorem{theorem}[proposition]{Theorem}
\newtheorem{lemma}[proposition]{Lemma}

\theoremstyle{definition}
\newtheorem{definition}[proposition]{Definition}

\newtheorem{remark}[proposition]{Remark}

\usepackage{geometry}
\usepackage{subeqnarray}
\usepackage{enumerate}
\usepackage{color}

\makeatletter
\DeclareRobustCommand\widecheck[1]{{\mathpalette\@widecheck{#1}}}
\def\@widecheck#1#2{%
	\setbox\z@\hbox{\m@th$#1#2$}%
	\setbox\tw@\hbox{\m@th$#1%
		\widehat{%
			\vrule\@width\z@\@height\ht\z@
			\vrule\@height\z@\@width\wd\z@}$}%
	\dp\tw@-\ht\z@
	\@tempdima\ht\z@ \advance\@tempdima2\ht\tw@ \divide\@tempdima\thr@@
	\setbox\tw@\hbox{%
		 \raise\@tempdima\hbox{\scalebox{1}[-1]{\lower\@tempdima\box
				\tw@}}}%
	{\ooalign{\box\tw@ \cr \box\z@}}}
\makeatother

\renewcommand{\widecheck}[1]{\widecheck{#1}}
\usepackage{bbm}

\newcommand{\RR}{\mathbb{R}}
\newcommand{\CC}{\mathbb{C}}
\newcommand{\NN}{\mathbb{N}}

\newcommand{\M}{\mathbf{M}}
\newcommand{\m}{\boldsymbol{m}}
\newcommand{\LL}{\mathbf{L}}

\newcommand{\G}{\mathbf{G}}
\renewcommand{\Re}{\text{Re}}
\let\on=\operatorname

\newenvironment{demo}[1]{\par\smallskip\noindent{\bf #1.}}{\par\smallskip}

\makeatletter
\@namedef{subjclassname@2020}{%
	\textup{2020} Mathematics Subject Classification}
\makeatother

\def\a{\alpha}
\def\b{\beta}

\definecolor{azulosc}{rgb}{0.2,0.1,0.7}
\definecolor{granate}{rgb}{0.6,0,0.3} 

\definecolor{verdeosc}{RGB}{46, 139, 87} 
\definecolor{rojo}{RGB}{219,0,0}                

\makeatletter
\DeclareRobustCommand\widecheck[1]{{\mathpalette\@widecheck{#1}}}
\def\@widecheck#1#2{%
    \setbox\z@\hbox{\m@th$#1#2$}%
    \setbox\tw@\hbox{\m@th$#1%
       \widehat{%
          \vrule\@width\z@\@height\ht\z@
          \vrule\@height\z@\@width\wd\z@}$}%
    \dp\tw@-\ht\z@
    \@tempdima\ht\z@ \advance\@tempdima2\ht\tw@ \divide\@tempdima\thr@@
    \setbox\tw@\hbox{%
       \raise\@tempdima\hbox{\scalebox{1}[-1]{\lower\@tempdima\box
\tw@}}}%
    {\ooalign{\box\tw@ \cr \box\z@}}}
\makeatother

\begin{document}

\title{Stability under product and composition for uniform Carleman asymptotic expansions}
\author{Javier Jiménez-Garrido \and Ignacio Miguel-Cantero \and Javier Sanz \and Gerhard Schindl}

\maketitle

\begin{abstract}
We study the stability under point-wise product and under composition in Carleman classes of holomorphic functions, defined on sectors of the Riemann surface of the logarithm, and admitting a uniform asymptotic expansion with remainders controlled by a given sequence of positive real numbers $\M$. On the one hand, the well-known conditions of algebrability and Fa\`a di Bruno, imposed on the sequence $\M$, ensure the desired stability with respect to each operation in both the Roumieu and the Beurling settings. On the other hand, these conditions turn out to be necessary for the corresponding stability in the Roumieu case as long as the existence of suitable characteristic functions, in a precise sense, is guaranteed within the class. The construction of such functions rests on classical results of B. Rodríguez-Salinas, and is given in detail. Our results are inspired by, and thoroughly generalize, several partial statements by G.~Auberson and G.~Mennessier for Gevrey classes of order 1. 
\end{abstract}

\par\medskip


\section{Introduction}

We denote by $\mathcal{\widetilde{A}}^{u}_{\{\M\}}(S)$ the ultraholomorphic Roumieu-type classes of functions with uniform asymptotic expansion, in a sector $S$ of the Riemann surface of the logarithm, whose remainders are bounded in terms of a weight sequence $\M = (M_j)_{j\in\mathbb{N}_0} \in \RR_{>0}^{\NN_0}$. This paper focuses on characterizing the stability of these classes under two fundamental operations: point-wise multiplication and composition. Our primary objective is to demonstrate that the stability properties of these classes can be characterized in terms of growth conditions on the defining sequence $\M$. For the problem of stability under point-wise multiplication, we show that the condition \hyperlink{s-alg}{$(\on{alg})$}, namely
$$\exists\;C\ge 1\;\forall\;j,k\in\NN_0:\;\;\;M_j M_k\le C^{j+k}M_{j+k},
$$
is sufficient to ensure that the class constitutes an algebra. Similarly, for the more complex problem of stability under composition, the \hyperlink{s-FdB}{$(\on{FdB})$} condition,
$$\exists\;C\ge 1\;\exists\;h\ge 1\;\forall\;k\in\NN \;\forall\;j_1,\dots,j_\ell\in\NN, \sum_{i=1}^{\ell}j_i=k:\;\;\;M_\ell\cdot M_{j_1}\cdots M_{j_{\ell}}\le Ch^kM_k,$$
serves as a sufficient requirement for closure. We mention that both statements apply to the Beurling-type classes as well.

The most challenging aspect of this work lies in proving that these sufficient conditions are, in fact, necessary in the Roumieu case. To achieve this characterization, we must ensure the existence of characteristic functions within the class. Following the ideas of B. Rodríguez-Salinas \cite{salinas62}, we define a function $f \in \widetilde{\mathcal{A}}^u_{\{\LL\}} (S)$ to be characteristic if its membership in a smaller class $\widetilde{\mathcal{A}}^u_{\{\M\}} (S) \subseteq \widetilde{\mathcal{A}}^u_{\{\LL\}} (S)$ implies the equality of the two spaces. As Theorem~\ref{optthm1} shows, this property is a consequence of the fact that the coefficients $(|c^f_j|)_{j\in\NN_0}$ of its asymptotic expansion are equivalent to the defining sequence~$\LL$.

The construction of these functions is achieved through the characteristic transform $\mathcal{T}_{\M}$, which allows us to modify some basic functions while maintaining precise control over the corresponding coefficients. By applying $\mathcal{T}_{\M}$ to these basic functions, we produce elements in $\widetilde{\mathcal{A}}^u_{\{\M\}} (S_{\alpha})$ whose coefficients $(|c_j^f|)_{j\in\NN_0}$ are equivalent to $\M$, where $S_{\alpha}$ stands for the unbounded sector bisected by the positive real line with opening $\pi\a$ in the Riemann surface of the logarithm. Furthermore, these coefficients are real, and we maintain precise control over their signs, which will be crucial in our reasoning. It should be emphasized that the requirements for the existence of characteristic functions are not uniform, varying significantly between sectors of small and large opening.

It is important to note that the proofs of both sufficiency and necessity rely on the combinatorial formulas for asymptotic expansions established by G. Auberson and G. Mennessier \cite{AubersonMennessier81}, which allow for the precise tracking of remainder terms and coefficient dependencies.

The stability properties of these classes have been a subject of extensive study in several related functional settings. In the ultradifferentiable context, A. Rainer and the fourth author \cite{compositionpaper, characterizationstabilitypaper} provided a complete characterization of these properties. In those works, the condition of derivation closedness, denoted \hyperlink{dc}{$(\on{dc})$}, is generally required for stability under composition; later, J. Pöschel \cite{Poschel} obtained results, in the weight sequences case, under the assumption of log-convexity \hyperlink{lc}{$(\on{lc})$} without requiring \hyperlink{dc}{$(\on{dc})$} (although some technical issue is to be fixed in his arguments regarding the composition of characteristic functions). It is worth noting that in the ultradifferentiable setting, the construction of characteristic functions is notably simpler than in the ultraholomorphic case. 

In the ultraholomorphic setting, earlier results often focused on specific weight sequences rather than general characterizations. For instance, G. Auberson and G. Mennessier \cite{AubersonMennessier81} proved the stability of Gevrey classes of order $1$ (corresponding to $M_j = j!$). Their result is not a characterization in the sense presented here, as Gevrey sequences inherently satisfy both \hyperlink{s-alg}{$(\on{alg})$} and \hyperlink{s-FdB}{$(\on{FdB})$} conditions. Furthermore, their technical approach to composition relied on specific combinatorial properties of factorials rather than the more general \hyperlink{s-FdB}{$(\on{FdB})$} condition. For the multivariate case of composition of asymptotic expansions, we refer to the work of J. Mozo \cite{MozoPhd}, whose results ensure stability under composition for the Gevrey classes of positive order $s$ (corresponding to $M_j = j!^s$).

Finally, in a previous work by the authors \cite{ultraholomstab}, characterization results were established for ultraholomorphic classes defined by uniform bounds on the derivatives. While those classes are closely related to the ones considered here, the arguments are fundamentally different, requiring now the construction of distinct characteristic functions and specialized technical machinery to handle the uniform nature of the asymptotic expansions.

It should be highlighted that the aforementioned results by A. Rainer and the fourth author were established for ultradifferentiable classes defined in terms of a weight matrix. This approach provides a unified framework that generalizes both the weight sequence and the weight function settings. It is expected that the characterization results obtained in this paper can be extended to ultraholomorphic classes defined in terms of weight matrices, considered in \cite{sectorialextensions, sectorialextensions1}. This generalization is currently part of an ongoing work in progress.

The structure of this work is organized as follows. We begin with a review of weight sequences and their properties, followed by the definition of the ultraholomorphic classes and the construction of the characteristic functions mentioned above. In Section~\ref{sec.formula}, we present the auxiliary formulas for the expansion of products and compositions. These results lead to Section~\ref{sec.alg}, where we characterize the closure under point-wise multiplication via the \hyperlink{s-alg}{$(\on{alg})$} condition. Finally, we address the closure under composition, establishing the equivalence between this stability property and the \hyperlink{s-FdB}{$(\on{FdB})$} condition.

\section{Weight sequences}\label{sect.WeightCondit}


We write $\NN_0:=\{0,1,2,\dots\}$ and $\NN:=\{1,2,3,\dots\}$. In what follows, we always denote by $\M=(M_j)_j\in\RR_{>0}^{\NN_0}$ a sequence with $M_0=1$.
We also use 
the sequence of quotients $\m=(m_j)_j$ defined by $m_j:=M_{j+1}/M_{j}$, $j\in \NN_0$.

$\M$ is said to be {\itshape log-convex}, (for short, \hypertarget{lc}{$(\text{lc})$}) if
$$\forall\;j\in\NN:\;M_j^2\le M_{j-1} M_{j+1},$$
equivalently if $\m$ is nondecreasing. If $\M$ is log-convex, 
then 
$j\mapsto(M_j)^{1/j}$ is nondecreasing and $(M_j)^{1/j}\le m_{j-1}$ for all $j\in\NN$. Finally,
\begin{equation}\label{eq.alg.from.lc}
M_j M_k\le M_{j+k},\quad j,k\in\NN_0.
\end{equation}

We say that a sequence $\M$ is a \emph{weight sequence} if it is \hyperlink{lc}{$(\text{lc})$} and $\lim_{j\to\infty} m_j =\infty$.

We shall use the following conditions on  sequences $\M$:

\begin{itemize}
	\item[\hypertarget{s-alg}{$(\on{alg})$}] $\M$ {is} {\itshape algebraic,} if
$$\exists\;C\ge 1\;\forall\;j,k\in\NN_0:\;\;\;M_j M_k\le C^{j+k}M_{j+k}.$$

\item[\hypertarget{s-FdB}{$(\on{FdB})$}] $\M$ has the {\itshape Fa\`{a}-di-Bruno property}, if
$$\exists\;C\ge 1\;\exists\;h\ge 1\;\forall\;k\in\NN \;\forall\;j_1,\dots,j_\ell\in\NN, \sum_{i=1}^{\ell}j_i=k:\;\;\;M_\ell\cdot M_{j_1}\cdots M_{j_{\ell}}\le Ch^kM_k.$$
\end{itemize}

\begin{remark}\label{newcondrem}
Each of these two conditions holds when $\mathbf{M}$ is log-convex (and $M_0=1$). Indeed, from~\eqref{eq.alg.from.lc} we have \hyperlink{s-alg}{$(\on{alg})$} with $C=1$, and \cite[Lemma 2.2 (1)]{compositionpaper} shows how to obtain \hyperlink{s-FdB}{$(\on{FdB})$} from \hyperlink{lc}{$(\on{lc})$}.
\end{remark}

For $a\in\RR$ we set
$$\G^a:=(j!^a)_{j\in\NN_0},\hspace{15pt}\overline{\G}^a:=(j^{ja})_{j\in\NN_0},$$
i.e. for $a>0$ the sequence $\G^a$ is the Gevrey-sequence of index $a$. Clearly $ \G^a$ and $\overline{\G}^a $ are weight sequences for any $a>0$ (by the convention $0^0:=1$).


$\M$ satisfies {\itshape derivation closedness}, denoted by \hypertarget{dc}{$(\text{dc})$}, if
$$\exists\;D\ge 1\;\forall\;j\in\NN_0:\;M_{j+1}\le D^{j+1} M_j\Longleftrightarrow m_{j}\le D^{j+1}.$$
In \cite{Komatsu73} this is condition $(M.2)'$. 

Let $\M,\LL\in\RR_{>0}^{\NN_0}$ be given {with arbitrary $M_0,L_0>0$}, we write $\M\hypertarget{preceq}{\subset}\LL$ if 
$$\sup_{j\in\NN}\left(M_j/L_j\right)^{1/j}<+\infty
$$ 
or, equivalently, if there exist $A,B>0$ such that $M_j\le AB^jL_j$ for every $j\in\NN_0$. We say $\M$ and $\LL$ are {\itshape equivalent}, denoted by $\M\hypertarget{approx}{\approx}\LL$, if $\M\hyperlink{preceq}{\subset}\LL$ and $\LL\hyperlink{preceq}{\subset}\M$.
Note that, in case $M_0=L_0=1$, equivalence amounts to $B^jM_j\le L_j\le C^jM_j$ for every $j\in\NN_0$ and suitable $B,C>0$.


Finally{,} we recall some useful elementary estimates,
\begin{equation}\label{Stirling}
	\forall\;j\in\NN:\;\;\;\frac{j^j}{e^j}\le j!\le j^j,
\end{equation}
which immediately imply that  
%
 $\G^a\hyperlink{approx}{\approx}\overline{\G}^a$ for any $a\in\RR$.

\section{Uniform asymptotic expansion classes}\label{sect.AsymptoticexpClass}

We introduce now the crucial classes under consideration in this paper, i.e., classes of functions that admit a uniform asymptotic expansion at the vertex of the sector where they are defined. We define spaces of both Roumieu and Beurling type, analogously as it was done for classes weighting the derivatives of the functions under consideration; for the ultraholomorphic setting see~\cite{ultraholomstab}, and~\cite{compositionpaper} for ultradifferentiable classes.

Recall that $\mathcal{R}$ stands for the Riemann surface of the logarithm. 
We will work with sectors in $\mathcal{R}$ of the form
$$
\{z\in\mathcal{R}\colon \beta<\arg(z)<\delta,\ 0<|z|<R\},
$$
where $\beta,\delta\in\RR$ with $\beta<\delta$, and $R\in(0,\infty]$. They can be seen as part of the complex plane whenever $\delta-\beta\le 2\pi$. They are called bounded if $R\in(0,\infty)$, unbounded otherwise.

For $\a>0$, we consider unbounded sectors bisected by direction 0 and opening $\a\pi$,
$$S_{\a}:=\{z\in\mathcal{R}:|\arg(z)|<\frac{\a\,\pi}{2}\}.$$
Let $T$ and $S$ be sectors in $\mathcal{R}$ with vertex at $0$. We say that $T$ is a \emph{subsector} of $S$ whenever $T\subset S$, and $T$ is a \emph{proper subsector} of  $S$ if $\overline{T}\subset S$ (where the closure of $T$ is taken in $\mathcal{R}$, and so the vertex of the sector is not under consideration).

%
%

We start by recalling the concept of uniform asymptotic expansion.

Let $\M$ be a sequence, $S\subseteq\mathcal{R}$ a sector and $h>0$. We define
$\mathcal{\widetilde{A}}^{u}_{\M,h}(S)$ as the space of $f\in\mathcal{H}(S)$ for which there exists a formal complex power series $\sum_{k=0}^{\infty}c_kz^k$ such that 
\[
\left\|f\right\|_{\M,h,\overset{\sim}{u}}:=\sup_{z\in S,j\in\NN_{0}}\frac{|f(z)-\sum_{k=0}^{j-1}c_kz^k|}{h^{j}M_{j}|z|^j}<+\infty.
\]
$(\mathcal{\widetilde{A}}^{u}_{\M,h}(S),\|\cdot\|_{\M,h,\overset{\sim}{u}})$ is a Banach space and alternatively we write $f\sim^u_{\M,h} \sum_{j=0}^\infty c_j z^j$ if $f\in\mathcal{\widetilde{A}}^{u}_{\M,h}(S)$ and say that $f$ admits $\sum_{k=0}^{\infty}c_kz^k$ as its \emph{uniform $\M$-asymptotic expansion of type $1/h$.} Next we put
\begin{equation*}
		 \mathcal{\widetilde{A}}^{u}_{\{\M\}}(S):=\bigcup_{h>0}\mathcal{\widetilde{A}}^{u}_{\M,h}(S),
\end{equation*}
and so $\widetilde{\mathcal{A}}^u_{\{\M\}}(S)$ stands for the $(LB)$ space of functions admitting a uniform $\{\M\}$-asymptotic expansion (of Roumieu type) in $S$. Occasionally, we write $f\sim^u_{\{\M\}} \sum_{j=0}^\infty c_j z^j$ if $f\in\widetilde{\mathcal{A}}^u_{\{\M\}}(S)$. Moreover, we can consider the space of Beurling-type, denoted by $\widetilde{\mathcal{A}}^u_{(\M)}(S)$, and defined as
$$
\widetilde{\mathcal{A}}^u_{(\M)}(S):=\bigcap_{h>0}\widetilde{\mathcal{A}}^u_{\M,h}(S),
$$
which becomes a Fr\'echet space when endowed with the topology generated by the family of seminorms $(\left\|\cdot\right\|_{\M,h,\overset{\sim}{u}})_{h>0}$. In this case, we write $f\sim^u_{(\M)} \sum_{j=0}^\infty c_j z^j$ if $f\in\widetilde{\mathcal{A}}^u_{(\M)}(S)$.

Let us fix some notation. If $f\in \widetilde{\mathcal{A}}^u_{\{\M\}} (S)$  with $f\sim^u_{\{\M\}} \sum_{j=0}^\infty c_j z^j$, then we denote the remainder
\begin{equation*}
	r_{f}(z,n):= f(z)-\sum_{j=0}^{n-1} c_j z^j,\;\;\;z\in S,\;n\in\NN_0,
\end{equation*}
and also observe that
\begin{equation*}
	\forall\;z\in S\;\forall\;n\in\NN_0:\;\;\; \frac{zr_f(z,n+1)}{z^{n+1}}=\frac{r_f(z,n)}{z^{n}}-c_n.
\end{equation*}
If $f\in\widetilde{\mathcal{A}}^u_{\{\M\}} (S)$, then $z\mapsto z^{-n}r_f(z,n)$ is bounded in $S$ for each $n\in\NN_0$ fixed. This implies
\begin{equation}\label{eq.coeff.asymp.expan.lim.remainder}
    \forall\;n\in\NN_0:\;\;\;\lim_{z\in S,z\rightarrow 0} z^{-n}r_f(z,n)=c_n.
\end{equation}
Moreover, if $f\sim^u_{\M,h} \sum_{j=0}^\infty c_j z^j$ on $S$, then by \eqref{eq.coeff.asymp.expan.lim.remainder} we have
\begin{equation}\label{cjestimation}
	\exists\;A>0\;\forall\;n\in\NN_0:\;\;\;|c_n|\le Ah^nM_n.
\end{equation}

It is also a well-known fact, steming from Cauchy's integral formula for the derivatives, that for every $n\in\NN_0$ and every proper subsector $T$ of $S$ one has
\begin{equation}\label{eq.limit.derivatives.coeff}
\lim_{\substack{z\to 0\\z\in T}}\frac{f^{(n)}(z)}{n!}=c_n.    
\end{equation}

\begin{remark}
%
	When a statement is valid for both Roumieu and Beurling classes, we will use the notation $\widetilde{\mathcal{A}}^u_{[\M]}(S)$ and so on (substituting every square bracket by either of them, curly brackets or parentheses, but the same all through the statement). Moreover, if $\M\approx\LL$ the corresponding classes coincide (as locally convex vector spaces), i.e. $\widetilde{\mathcal{A}}^u_{[\M]}(S)=\widetilde{\mathcal{A}}^u_{[\LL]}(S)$.
	
\end{remark}	
\section{Characteristic functions for classes with Roumieu uniform asymptotics}\label{sect.CharactFunctions}

We introduce the concept of characteristic functions for the classes under consideration.

\begin{definition}\label{chardef}
	Let $\LL\in\RR_{>0}^{\NN_0}$ and $S$ be a given sector.
A function $f\in \widetilde{\mathcal{A}}^u_{\{\LL\}} (S)$ is said to be \textit{characteristic} in the class $\widetilde{\mathcal{A}}^u_{\{\LL\}} (S)$ if, whenever $f\in \widetilde{\mathcal{A}}^u_{\{\M\}} (S)\subseteq \widetilde{\mathcal{A}}^u_{\{\LL\}} (S)$ for some $\M\in\RR_{>0}^{\NN_0}$, we have that $\widetilde{\mathcal{A}}^u_{\{\M\}} (S)= \widetilde{\mathcal{A}}^u_{\{\LL\}} (S)$.
\end{definition}



Let $f\in \widetilde{\mathcal{A}}^u_{\{\LL\}} (S)$. We consider the sequence defined by
\begin{equation*}
	\widetilde{C}_n(f):=\sup_{z\in S}|z^{-n}r_f(z,n)|, \qquad n\in\NN_0.
\end{equation*}

We deduce that a function is characteristic of the class if the sequence of moduli of the coefficients of the associated asymptotic expansion (resp. the sequence $(\widetilde{C}_j(f))_{j\in\NN_0}$) is equivalent to the sequence defining the class. More precisely,

\begin{theorem}\label{optthm1}
	Let $\LL\in\RR_{>0}^{\NN_0}$, $S$ be a given sector and $f\in \widetilde{\mathcal{A}}^u_{\{\LL\}} (S)$ with $f\sim^u_{\{\LL\}} \sum_{j=0}^\infty c_j z^j$. Then, each of the following  conditions implies the next one:
	\begin{enumerate}
		\item[$(1)$] The sequence $(|c_j|)_{j\in\NN_0}$ is equivalent to $\LL$.
		\item[$(2)$] The sequence $(\widetilde{C}_j(f))_{j\in\NN_0}$ is equivalent to $\LL$.
		\item[$(3)$] $f$ is characteristic in the class $\widetilde{\mathcal{A}}^u_{\{\LL\}} (S)$.
	\end{enumerate}
\end{theorem}

\demo{Proof}
(1) $\Rightarrow$ (2) As $f\in \widetilde{\mathcal{A}}^u_{\{\LL\}} (S)$, there exist $A,B>0$ such that $\widetilde{C}_n(f)\le AB^nL_n$ for every $n\in\NN_0$. On the other hand, it is clear that $\widetilde{C}_n(f)\geq |c_n|$, and the hypothesis allows us to conclude the other estimate.

(2)  $\Rightarrow$ (3) By assumption, there exist $A,B>0$ such that $L_n\le A{B}^n\widetilde{C}_n(f)$ for every $n\in\NN_0$. If for some $\M=(M_n)_{n\in\NN_0}\in\RR_{>0}^{\NN_0}$ we have $f\in \widetilde{\mathcal{A}}^u_{\{\M\}} (S)\subseteq \widetilde{\mathcal{A}}^u_{\{\LL\}} (S)$, there exist $C,D>0$ such that $\widetilde{C}_n(f)\le C {D}^nM_n$ for every $n\in\NN_0$. The two deduced inequalities show that $L_n\le{ AC(BD)^n}M_n$ for every $n\in\NN_0$, what easily implies that $\widetilde{\mathcal{A}}^u_{\{\LL\}} (S)\subseteq \widetilde{\mathcal{A}}^u_{\{\M\}} (S)$, and we are done.
\qed\enddemo

%
%
%

\subsection{Basic functions}\label{ssect.BasicFct}

Recall the notations {$\G^s:=(j!^s)_{j\in\NN_0}$ and }$\overline{\G}^s:=(j^{js})_{j\in\NN_0}$, $s\in\RR$, and that $\overline{\G}^s\hyperlink{approx}{\approx}\G^s$, see \eqref{Stirling}. We are going to construct characteristic functions in the uniform asymptotic classes associated with $\overline{\G}^s$, where $s$ depends on the opening of the sector.


The \textit{{two-parametric} Mittag-Leffler function} is defined for all {complex parameters $A, B$ with $\Re(A)>0$} by
$$E_{A,B}(z):=\sum_{j=0}^\infty \frac{z^j}{\Gamma (Aj+B)},\quad z\in\CC,$$
where $\Gamma$ denotes the Gamma function.
For the construction of characteristic functions in sectors $S_\alpha$ for $\alpha\in(0,1]$ we will take $A=2-\alpha$ and $B=4-\alpha$ and we set
\begin{equation*}
	\widetilde{E}_{\alpha}(z):=E_{2-\alpha, 4-\alpha}(-z){=\sum_{j=0}^\infty \frac{(-1)^jz^j}{\Gamma ((2-\alpha)(j+1)+2)},\quad z\in\CC}.
\end{equation*}

We recall the following statement.

\begin{theorem}\label{Ealphathm} (\cite[Thm. 5, Thm. 20]{salinas62})
	Let $\alpha\in(0,2)$, then
	\begin{equation*}
		\exists\;C\ge 1\;\forall\;z\in S_{\alpha}\;\forall\;n\in\NN_0:\;\;\;\left| \widetilde{E}_{\alpha}(z)-\sum_{j=0}^{n-1} \frac{(-z)^j}{\Gamma ((j+1) (2-\alpha)+2)}\right|\leq C \frac{e^n}{n^{(2-\alpha)n}}|z|^n.
	\end{equation*}
	Consequently, $\widetilde{E}_{\alpha}\in\widetilde{\mathcal{A}}^u_{\{\overline{\G}^{\alpha-2}\}}(S_\alpha)$ and,  
    moreover, 
    $\widetilde{E}_{\alpha}$ is a characteristic function in the class $\widetilde{\mathcal{A}}^u_{\{\overline{\G}^{\alpha-2}\}}(S_\alpha)$.
\end{theorem}

Next we consider the Rodr\'{i}guez-Salinas function (see \cite{salinas62})
\begin{equation*}
	K_{RS}(z):=\frac{(1+z)\log(1+z)-z}{z^2},
\end{equation*}
which is clearly holomorphic in $S_2$.

This function satisfies the following theorem 

\begin{theorem}\label{RStheorem} (\cite[Thm. 3]{salinas62})
	There exists a constant $C$ (we can take $C=4$) such that
	 $$\forall\;z\in S_2
	 \;\forall\;n\in\NN_0:\;\;\;\left| K_{RS}(z)-\sum_{j=0}^{n-1} \frac{(-z)^j}{(j+1)(j+2)}\right|\leq C|z|^n.$$
	Consequently, for the constant sequence $\mathbf{1}:=(1)_{j\in\NN_0}$ we have that  $K_{RS}\in\widetilde{\mathcal{A}}^u_{\{\mathbf{1}\}}(S_2)$ and, moreover, $K_{RS}$ is a characteristic function in the class $\widetilde{\mathcal{A}}^u_{\{ \mathbf{1}\}}(S_2)$.
\end{theorem}
Finally, let $\alpha>2$ and take $\alpha'>\alpha$. For all $z\in S_\alpha$ we define
\begin{equation}\label{fctfalphaalphaprime}
	f_{\alpha,\alpha'}(z):= \int_{0}^{\infty(-\phi)} K_{RS}(zv^{\alpha'-2}) e^{-v} dv,
\end{equation}
where we choose $\phi  \in (-\frac{(\alpha-2)}{(\alpha'-2)}\frac{\pi}{2}, \frac{(\alpha-2)}{(\alpha'-2)}\frac{\pi}{2})$ with $|\arg(z)- (\alpha'-2)\phi|<\pi$.\\
\begin{theorem}\label{fgalphaalphaprimethm} (\cite[Thm. 4, Thm. 8]{salinas62})
	Let $\alpha>2$, $\alpha'>\alpha$ and
	$f_{\alpha,\alpha'}$ be the function from
	\eqref{fctfalphaalphaprime}. Then
				\begin{align*}
						\exists\;C,A\ge 1\;\forall\;z\in S_{\alpha}\;\forall\;n\in\NN_0:\;\;\;&\left|f_{\alpha,\alpha'} (z) -\sum_{j=0}^{n-1} \frac{(-z)^j}{(j+1)(j+2)} \Gamma (  (\alpha'-2)j+1)\right|\nonumber
						\\&
						\le CA^n \Gamma ( (\alpha'-2)n+1)|z|^n.
					\end{align*}
				Consequently, $f_{\alpha,\alpha'}\in\widetilde{\mathcal{A}}^u_{\{\overline{\G}^{\alpha'-2}\}}(S_\alpha)$ and,
				moreover, $f_{\alpha,\alpha'}$ is a characteristic function in the class $\widetilde{\mathcal{A}}^u_{\{ \overline{\G}^{\alpha'-2} \}}(S_\alpha)$.
	\end{theorem}

\subsection{Characteristic transform}\label{ssect.CharactTransf}
%
%
The $\mathcal{T}_{\M}$ transform originally appeared in the work of Rodríguez Salinas~\cite{salinas62} and was used by the authors in the previous work~\cite{ultraholomstab} to construct characteristic functions with precise control of the derivatives at $0$ in ultraholomorphic classes defined by uniform bounds on the derivatives.

In this context, the same functional transform $\mathcal{T}_{\M}$ modifies the coefficients of the asymptotic expansion of a function $f\in\widetilde{\mathcal{A}}^u_{\{\LL\}} (S)$ to obtain a new function in the class $\widetilde{\mathcal{A}}^u_{\{\LL\M\}} (S)$, under suitable hypotheses. This transformation allows us to construct characteristic functions in classes more general than the Gevrey classes considered previously.
\begin{definition}
	Let $\M$ be an \hyperlink{lc}{$(\text{lc})$} sequence, $\LL\in\RR_{>0}^{\NN_0}$, $S$ a sector and $f\in\widetilde{\mathcal{A}}^u_{\{\LL\}} (S)$. 
Then we define the $\mathcal{T}_{\M}-$transform of $f$ by
	\begin{equation*}
		\mathcal{T}_{\M}(f)(z):= \sum_{j=0}^{\infty}\frac{1}{2^j} \frac{M_j}{m_j^j} f(m_j z), \qquad z\in S.
	\end{equation*}
\end{definition}
This expression should be compared with the characteristic functions obtained in the ultradifferentiable and ultraholomorphic setting in \cite[Thm. 1]{thilliez}, \cite[Lemma 2.9]{compositionpaper}, \cite[Lemma 1]{Poschel} and \cite[Definition 4.5.]{ultraholomstab}. 
For every $j\in\NN_0$ let us set
\begin{equation*}
	R_j:=\sum^{\infty}_{n=0} \frac{1}{2^n} \frac{M_n}{m_n^n}m_n^j.
\end{equation*}
The following result provides estimates for this sequence in terms of the general sequence $\M$ we depart from.

\begin{lemma}\label{Scomplemma}(\cite[Lemma 4.6]{ultraholomstab})
	Let $\M\in\RR_{>0}^{\NN_0}$, then
	$$\forall\;j\in\NN_0:\;\;\;R_j\ge\frac{1}{2^j} M_j.$$
	If $\M$ is \hyperlink{lc}{$(\text{lc})$}, then also
	$$\forall\;j\in\NN_0:\;\;\;R_j\le 2 M_j,$$
	and so $(R_j)_{j\in\NN_0}$ is equivalent to $\M$.
\end{lemma}



\begin{theorem}\label{CharacteristicTransformUniformAsymptotic}
        Let $\M$ be a \hyperlink{lc}{$(\text{lc})$} sequence, $\LL\in\RR_{>0}^{\NN_0}$ and for a given sector $S$ take $f\in\widetilde{\mathcal{A}}^u_{\{\LL\}} (S)$,  with $f\sim^u_{\{\LL\}}\sum_{j=0}^\infty c_j z^j$. Then, $\mathcal{T}_{\M} (f)\in \widetilde{\mathcal{A}}^u_{\{\LL\M\}} (S)$  with 
        \begin{equation}\label{eq.CoeffT_Mf}
        	\mathcal{T}_{\M} (f)\sim^u_{\{\LL\M\}} \sum_{j=0}^{\infty}R_jc_j z^j,\qquad z\in S.
        \end{equation}
        Moreover, for any $A>0$, $\mathcal{T}_{\M} : \widetilde{\mathcal{A}}^u_{\LL, A} (S)\rightarrow\widetilde{\mathcal{A}}^u_{\LL\M, A } (S)$ is a continuous linear operator.
\end{theorem}
\demo{Proof}
By definition of $\widetilde{\mathcal{A}}^u_{\{\LL\}} (S)$ we have that $f$ is bounded in $S$ by some constant $C>0$. Since $\M$ is log-convex, we have that $M_j\le m_j^j$ for all $j\in\NN_0$ and then
$$\sum_{j=0}^{\infty}\frac{1}{2^j} \frac{M_j}{m_j^j} \left|f(m_j z) \right| \leq  C  \sum_{j=0}^{\infty}\frac{1}{2^j}=2C, \qquad z\in S.$$
Consequently, the series defining $\mathcal{T}_{\M} (f)$ normally converges in the whole of $S$, it provides a function holomorphic in $S$.
For each $z\in S$ and every $n\in\NN_0$, we observe then that
\begin{align*}
	\left|\mathcal{T}_{\M} (f) (z) -  \sum_{k=0}^{n-1} R_kc_k z^k\right|&= \left|\sum_{j=0}^{\infty}\frac{1}{2^j} \frac{M_j}{m_j^j} f(m_j z) -  \sum_{k=0}^{n-1} c_k  (\sum^{\infty}_{j=0} \frac{1}{2^j} \frac{M_j}{m_j^j}m_j^k)  z^k\right|
	\\&
	=\left|\sum_{j=0}^{\infty}\frac{1}{2^j} \frac{M_j}{m_j^j}  \left(f(m_j z) -  \sum_{k=0}^{n-1} c_k   (m_jz)^k\right)\right|
	\\&
	\le\sum_{j=0}^{\infty}\frac{1}{2^j}\frac{M_j}{m_j^j}C A^n L_n m_j^n |z|^n= C A^n L_n R_n |z|^n
	\\&
	\le 2 C A^n L_n M_n |z|^n,
\end{align*}
where the last estimate holds by Lemma \ref{Scomplemma}.
Finally, suppose $f\in \widetilde{\mathcal{A}}^u_{\LL,A} (S)$ for some $A>0$, then for all $n\in\NN_0$ we can estimate
\begin{align*}
	&\left|\mathcal{T}_{\M} (f) (z) -  \sum_{k=0}^{n-1} R_kc_k z^k\right|\leq \sum_{j=0}^{\infty}\frac{1}{2^j} \frac{M_j}{m_j^j}  \left|f(m_j z) -  \sum_{k=0}^{n-1} c_k   (m_jz)^k\right|
	\\&
	\le \|f\|_{\M,A,\overset{\sim}{u}} A^n L_n|z|^n \sum_{j=0}^{\infty}\frac{1}{2^j}M_jm_j^{n-j}=\|f\|_{\M,A,\overset{\sim}{u}} A^n L_n R_n|z|^n.
\end{align*}
By Lemma~\ref{Scomplemma} we know that $R_n\le 2M_n$, so $\mathcal{T}_{\M} (f)\in \mathcal{A}_{\LL\M, A} (S)$, and moreover
$$\|\mathcal{T}_{\M} (f)\|_{\LL\M, A,\overset{\sim}{u}}=\sup_{z\in S,n\in\NN_{0}}\frac{|\mathcal{T}_{\M} (f) (z) -  \sum_{k=0}^{n-1} R_kc_k z^k|}{A^{n}L_{n}M_{n}|z|^n}\le 2\|f\|_{\M,A,\overset{\sim}{u}}.$$
It follows that $\mathcal{T}_{\M} : \widetilde{\mathcal{A}}^u_{\LL, A} (S)\rightarrow\widetilde{\mathcal{A}}^u_{\LL\M, A } (S) $ is a well-defined continuous linear operator for any $A>0$.
\qed\enddemo

\begin{theorem}\label{CharacteristicTransformUniformAsymptotic1}
 Let $\M$ be a \hyperlink{lc}{$(\text{lc})$} sequence, $\LL\in\RR_{>0}^{\NN_0}$ and for a given sector $S$ take $f\in \widetilde{\mathcal{A}}^u_{\{\LL\}} (S)$  with $f\sim^u_{\{\LL\}}\sum_{j=0}^{\infty}c_j z^j$. If $(|c_j|)_{j\in\NN_0}$ is equivalent to $\LL$, then $(|c_j R_j|)_{j\in\NN_0}$ is
 equivalent to $\LL\M$. Consequently, $\mathcal{T}_{\M}(f)$ is characteristic in the class $\widetilde{\mathcal{A}}^u_{\{\LL\M\}} (S)$.
\end{theorem}

\demo{Proof}
The first assertion is clear from Lemma~\ref{Scomplemma}. The second one stems from Theorems~\ref{CharacteristicTransformUniformAsymptotic} and~\ref{optthm1}.
\qed\enddemo

\subsection{Construction of characteristic functions}\label{ssect.ConstrCharactFunct}
Given a sequence $\M\in\RR_{>0}^{\NN_0}$ and $\alpha>0$ we construct now, under suitable assumptions, characteristic functions in $\widetilde{\mathcal{A}}^u_{\{\M\}} (S_{\alpha})$.
For this we are using the basic functions from Subsection \ref{ssect.BasicFct} and the characteristic transform from Subsection \ref{ssect.CharactTransf}.
%
%
%




\begin{theorem}\label{mainTthm1}
	Let $\M\in\RR_{>0}^{\NN_0}$ and $\alpha>0$.
	\begin{enumerate}
		\item If $\alpha< 2$, we assume that  $\overline{\G}^{2-\alpha}\M:=(j^{(2-\alpha)j}M_j)_{j\in\NN_0}$
		is equivalent to an \hyperlink{lc}{$(\text{lc})$} sequence $\LL$. Then, $\mathcal{T}_{\LL} (\widetilde{E}_{\alpha}) $ is  characteristic in the class $\widetilde{\mathcal{A}}^u_{\{\M\}} (S_{\alpha})$.
		\item If $\alpha= 2$, we assume that  $\M$
		is equivalent to an \hyperlink{lc}{$(\text{lc})$} sequence $\LL$. Then, $\mathcal{T}_{\LL} (K_{RS}) $ is characteristic in the class $\widetilde{\mathcal{A}}^u_{\{\M\}} (S_{2})$.
		\item If $\alpha>2$, we assume that there exists $\a'>\a$ such that $\overline{\G}^{2-\alpha'}\M:=(j^{(2-\alpha')j}M_j)_{j\in\NN_0}$
		is equivalent to an \hyperlink{lc}{$(\text{lc})$} sequence $\LL$. Then, $\mathcal{T}_{\LL} (f_{\alpha,\alpha'}) $ is characteristic in the class $\widetilde{\mathcal{A}}^u_{\{\M\}} (S_{\alpha})$.
	\end{enumerate}
\end{theorem}
\demo{Proof}
This follows by Theorems~\ref{Ealphathm}, \ref{RStheorem}, \ref{fgalphaalphaprimethm},  \ref{CharacteristicTransformUniformAsymptotic} and \ref{CharacteristicTransformUniformAsymptotic1}, and  from the fact that $\overline{\G}^{\alpha-2}\LL$ in case 1, resp. $\LL$ in case 2, resp. $\overline{\G}^{\alpha'-2}\LL$ in case 3, is equivalent to $\M$.
\qed\enddemo
\begin{remark}
	As it was carefully explained in~\cite[Remark 4.10]{ultraholomstab}), in order to guarantee that the hypotheses in the previous theorem are satisfied one can compute the lower Matuszewska index  $\gamma(\M)$, associated to the sequence of quotients $\m$ in the general theory of $O$-regular variation of sequences, and check whether it is greater than $\a-2$.
\end{remark}


\begin{remark}\label{rem.signofcoefficients} For future use, observe that for any of the characteristic functions $f$ constructed in Theorem~\ref{mainTthm1}, if  $f\sim^u_{\{\M\}} \sum_{j=0}^\infty c_j z^j$, the coefficients are real and have alternating signs, in fact they satisfy $c_j=(-1)^j|c_j|$.
\end{remark}

\section{Auxiliary formulas for point-wise multiplication and composition of asymptotic expansion}\label{sec.formula}

Following the approach in \cite{AubersonMennessier81}, we establish formulas describing the behavior of coefficients and remainders under point-wise multiplication and composition for functions, admitting a uniform asymptotic expansion at the point 0. 
Since we are dealing with different functions, in order to emphasize the dependence on the considered function we write for the appearing coefficients in the asymptotic expansions $c_j^f$, $c_j^g$ and so on. 

On the one hand, let $f$ and $g$ be functions defined in an arbitrary sector $S\subset \mathcal{R}$, and write $f(z)=\sum_{j=0}^{n-1}c^f_jz^j+r_f(z,n)$ and $g(z)=\sum_{j=0}^{n-1}c^g_jz^j+r_g(z,n)$. Then for the  point-wise multiplication $h:=f\cdot g$ we get $h(z)=\sum_{j=0}^{n-1}c^h_jz^j+r_h(z,n)$ with
\begin{equation}\label{eq.stablepointwisepropequ1}
c^h_j=\sum_{k=0}^jc^f_kc^g_{j-k},\hspace{15pt}r_h(z,n)=r_f(z,n)g(z)+\sum_{k=0}^{n-1}c^f_kz^kr_g(z,n-k);
\end{equation}
see also \cite[$(33)$]{AubersonMennessier81}.  

We also include, following~\cite[$(A.3)$]{AubersonMennessier81}, a version of the second equality in~\eqref{eq.stablepointwisepropequ1} for the case that $h=f^j=f\cdot f^{j-1}$ for a function $f$ such that $c_0^f=0$, and for any $n>j$. We write
\begin{align*}
(f(z))^j&=f(z)(f(z))^{j-1}=\left(\sum_{i=1}^{n-j}c^f_i z^i+r_f(z,n-j+1)\right)(f(z))^{j-1}\\
		&=\sum_{i=1}^{n-j}c^f_i z^i\left[\sum_{\ell=j-1}^{n-i-1}c^{f^{j-1}}_{\ell}z^{\ell}+r_{f^{j-1}}(z,n-i)\right]+r_f(z,n-j+1)(f(z))^{j-1}.
	\end{align*}
Since $f(z)=r_f(z,1)$, one deduces that
	\begin{equation}\label{eq.remaindern.fpowerj}
		 r_{f^j}(z,n)=\sum_{i=1}^{n-j}c^f_iz^ir_{f^{j-1}}(z,n-i)+r_f(z,n-j+1)\left(r_f(z,1)\right)^{j-1}.
	\end{equation}

On the other hand, let $f$ be a function defined in an arbitrary sector $S\subset \mathcal{R}$ with $c_0^f=0$, and $g$ be a function defined in a region $T\subset \CC$ containing the range of $f$. For $\varphi:=g\circ f$ we can write for all $z\in S$:
	$$\varphi(z)=c^g_0+\sum_{k=1}^{n-1}c^g_k(f(z))^k+r_g(f(z),n).$$
	Moreover, we have that
	\begin{equation*}
		(f(z))^k=\sum_{j=k}^{n-1}c^{f^k}_jz^j+r_{f^k}(z,n),
	\end{equation*}
	with
	\begin{equation}\label{equ71}
		c^{f^k}_j=\sum_{m_1+\dots+m_k=j}c^{f}_{m_1}\cdots c^{f}_{m_k},
	\end{equation}
	see \cite[$(70)$ \& $(71)$]{AubersonMennessier81}. Next we recall the crucial formulas \cite[$(73)$, $(74)$ \& $(75)$]{AubersonMennessier81}:
	\begin{equation*}
		\varphi(z)=\sum_{j=1}^{n-1}c^\varphi_jz^j+r_{\varphi}(z,n),
	\end{equation*}
	\begin{equation}\label{equ74}
		c^\varphi_0=c^g_0,\hspace{15pt}c^\varphi_n=\sum_{j=1}^{n}c^g_jc^{f^j}_n,\;\;\;n\in\NN,
	\end{equation}
	and
	\begin{equation}\label{eq.equ75.roumieu}
		r_\varphi(z,n)=\sum_{j=1}^{n-1}c^g_jr_{f^j}(z,n)+r_g(f(z),n).
	\end{equation}

\section{Stability under point-wise multiplication}\label{sec.alg}

We start with the following generalization of \cite[Thm. 3]{AubersonMennessier81}:

\begin{proposition}\label{prop.stablepointwiseprop}
Let $\M=(M_j)_j\in\RR_{>0}^{\NN_0}$ be a sequence with $M_0=1$, and let $S$ be a sector. If $\M$ satisfies \hyperlink{s-alg}{$(\on{alg})$}, then the class $\mathcal{\widetilde{A}}^{u}_{[\M]}(S)$ is closed under point-wise multiplication of functions; i.e. $\mathcal{\widetilde{A}}^{u}_{[M]}(S)$ is an algebra. 
\end{proposition}

\begin{proof} Since $\M$ satisfies \hyperlink{s-alg}{$(\on{alg})$}, we have $M_jM_k\le C^{j+k}M_{j+k}$ for some $C\geq 1$ and for all $j,k\in\NN_0$.
Let $f,g\in\mathcal{\widetilde{A}}^{u}_{[\M]}(S)$.  Now, by the definition of the classes, we have $f\in\mathcal{\widetilde{A}}^{u}_{\M,d_1}(S)$ and $g\in\mathcal{\widetilde{A}}^{u}_{\M,d_2}(S)$ for some (resp. all) $d_i>0$, and so there exist $A_1,A_2>0$ such that
$$
|r_f(z,n)|\le  A_1 d_1^n M_n|z|^n,\quad 
|r_g(z,n)|\le  A_2 d_2^n M_n|z|^n,\quad z\in S.
$$ 
Thus by \eqref{eq.stablepointwisepropequ1} and \eqref{cjestimation}, we estimate as follows for all $n\in\NN_0$ and $z\in S$:

\begin{align*}
|r_h(z,n)|&\le |r_f(z,n)||g(z)|+\sum_{k=0}^{n-1}|c^f_k||z|^k|r_g(z,n-k)|\\&
\le A_1 d_1^n M_n|z|^n A_2 + \sum_{k=0}^{n-1} A_1 d_1^k M_k|z|^k A_2 d_2^{n-k}M_{n-k}|z|^{n-k}
\\&
\le A_1A_2d_1^nM_n|z|^n + A_1A_2\sum_{k=0}^{n-1} d_1^k d_2^{n-k}C^nM_n|z|^n
\\&
\le A_1A_2d_1^nM_n|z|^n+A_1A_2C^nM_n|z|^n\sum_{k=0}^{n}\binom{n}{k}d_1^kd_2^{n-k}
\\&
\le A_1A_2d_1^nM_n|z|^n+ A_1A_2C^nM_n|z|^n(d_1+d_2)^n
\\&
\le 2A_1A_2(C(d_1+d_2))^nM_n|z|^n.
\end{align*}
This finishes the Roumieu part. Concerning the Beurling class $\mathcal{\widetilde{A}}^{u}_{(\M)}$ we use the same estimate, and notice that $C(d_1+d_2)\rightarrow 0$ when $d_i\rightarrow 0$ since $C$ is only depending on $\M$ via \hyperlink{s-alg}{$(\on{alg})$} but not on $d_i$, $i=1,2$.
\end{proof}

Conversely, if we can ensure the existence of very well-suited characteristic functions in the Roumieu class, closure under point-wise multiplication implies condition \hyperlink{s-alg}{$(\on{alg})$}. The following result gathers different implications.

\begin{theorem}\label{thm.stablepointwiseprop}
Let $\M=(M_j)_j\in\RR_{>0}^{\NN_0}$ be a sequence with $M_0=1$ and $\alpha>0$. Consider the following statements:
\begin{itemize}
    \item[(i)] In case $\alpha\leq 2$, $\overline{\G}^{2-\alpha}\M$ is equivalent to an \hyperlink{lc}{$(\text{lc})$} sequence; if $\alpha>2$, there exists $\a'>\a$ such that $\overline{\G}^{2-\alpha'}\M$ is equivalent to an \hyperlink{lc}{$(\text{lc})$} sequence.
\item[(ii)] $\M$ satisfies \hyperlink{s-alg}{$(\on{alg})$}.
\item[(iii)] $\mathcal{\widetilde{A}}^{u}_{\{\M\}}(S_{\a})$ is an algebra.
\item[(iv)] There exists a characteristic function $f\in \mathcal{\widetilde{A}}^{u}_{\{\M\}} (S_\alpha)$, with  $f\sim^u_{\{\M\}} \sum_{j=0}^\infty c^f_j z^j$, such that $(|c^f_j|)_{j\in\NN_0}$ is equivalent to $\M$ and, moreover,  $c^f_j=(-1)^j|c^f_j|$ for all $j\in\NN_0$.
\end{itemize} 
Then, one has $(i)\Rightarrow(ii)\Rightarrow(iii)$, $(i)\Rightarrow(iv)$ and $(iii)+(iv)\Rightarrow(ii)$. 
\end{theorem}
\begin{proof}
$(i)\Rightarrow(ii)$ This a consequence of the following general statement, with a straightforward proof: Assume that we have sequences $\mathbf{M}$, $\mathbf{N}$ (with $M_0=1=N_0$) such that $\mathbf{M}\cdot\mathbf{N}\approx\mathbf{L}$ and $\mathbf{L}$ satisfies \hyperlink{s-alg}{$(\on{alg})$} (and $L_0=1$). If $\mathbf{N}$ satisfies the condition of moderate growth, namely there exists $A>0$ such that $N_{j+k}\le A^{j+k}N_jN_k$ for all $j,k\in\NN_0$, then $\mathbf{M}$ satisfies \hyperlink{s-alg}{$(\on{alg})$} as well. In our case, it suffices to observe that \hyperlink{lc}{$(\text{lc})$} implies \hyperlink{s-alg}{$(\on{alg})$} and that the Gevrey-like sequences $\overline{\G}^{\beta}$ of every order $\beta\in\RR$ satisfy the moderate growth condition.


$(ii)\Rightarrow(iii)$ This is  Proposition~\ref{prop.stablepointwiseprop}.

$(i)\Rightarrow(iv)$ See Theorem~\ref{mainTthm1} and Remark~\ref{rem.signofcoefficients}.

$(iii)+(iv)\Rightarrow(ii)$ Let $f\in \mathcal{\widetilde{A}}^{u}_{\{\M\}} (S_\alpha)$ be as in $(iv)$. Since $\mathcal{\widetilde{A}}^{u}_{\{\M\}} (S_\alpha)$ is an algebra we have that $h= f\cdot f \in \mathcal{\widetilde{A}}^{u}_{\{\M\}} (S_\alpha)$.

By \eqref{cjestimation}, there exist $A_1>0$ and $d_1>0$ such that $|c^h_n|\le A_1d_1^nM_n$ for all $n\in\NN_0$. Using the auxiliary formula for the coefficients of a product \eqref{eq.stablepointwisepropequ1}, we have $c^h_n = \sum_{k=0}^n c^f_k c^f_{n-k}$. So, for any $j\in\{0,1,2,\dots,n\}$ we see that
$$|c^h_n|=|\sum_{k=0}^n c^f_k \cdot c^f_{n-k}| = |\sum_{k=0}^n  (-1)^k|c^f_k| \cdot (-1)^{n-k} |c^f_{n-k}|| = \sum_{k=0}^n |c^f_k| \cdot |c^f_{n-k}|  \ge |c^f_j| \cdot |c^f_{n-j}|.$$
Since $(|c^f_j|)_{j\in\NN_0}$ is equivalent to $\M$, there exist $A_2,d_2>0$ such that $|c^f_j|\geq A_2 d_2^j M_j$ and for all $j,k\in\NN_0$, by taking $n=j+k$, we deduce that 
 $$A_1d_1^{j+k}M_{j+k} \geq |c^h_{j+k}| \ge |c^f_j| \cdot |c^f_{k}|\ge A_2^2 d_2^{j+k} M_j M_k.$$
 Consequently, $\M$ satisfies \hyperlink{s-alg}{$(\on{alg})$}.
\end{proof}

\section{Stability under composition}\label{sect.StabilityCompos}

Let us first define a suitable concept of stability under composition for classes associated with the sequence $\M$.

\begin{definition}\label{def.UniformAsympComplexPlane}
Let $V$ be an open subset in the complex plane such that $0\in\overline{V}$, $f\colon V\to\CC$ be a holomorphic function on $V$, $\M=(M_j)_j\in\RR_{>0}^{\NN_0}$ be a sequence with $M_0=1$, and $h>0$. We say $f\in\widetilde{\mathcal{A}}_{\M,h}^{u}(V)$ if there exists a formal complex power series $\sum_{k=0}^{\infty}c_kz^k$  such that
\[
\sup_{z\in V\setminus\{0\},\,j\in\NN_{0}}\frac{|f(z)-\sum_{k=0}^{j-1}c_kz^k|}{h^{j}M_{j}|z|^j}<+\infty.
\]   
We set 
$$\widetilde{\mathcal{A}}_{\{\M\}}^{u}(V):=\bigcup_{h>0}\widetilde{\mathcal{A}}_{\M,h}^{u}(V),\qquad
\widetilde{\mathcal{A}}_{(\M)}^{u}(V):=\bigcap_{h>0}\widetilde{\mathcal{A}}_{\M,h}^{u}(V).$$
\end{definition}

Observe that, in case $V$ is a sector $S$ with vertex at 0 in the complex plane, this definition agrees with that given for $\widetilde{\mathcal{A}}_{[\M]}^{u}(S)$
in Section~\ref{sect.AsymptoticexpClass}. 

\begin{definition}\label{def.StabilityCompos}
Given a sequence $\M$ and a sector $S$, we say the class $\widetilde{\mathcal{A}}_{[\M]}^{u}(S)$ is stable under composition if for every $f\in\widetilde{\mathcal{A}}_{[\M]}^{u}(S)$ with $c^f_0=0$ (with the notation in Section~\ref{sec.formula}), for every subsector $T$ of $S$ and for every $g\in\widetilde{\mathcal{A}}_{[\M]}^{u}(f(T))$ one has that $g\circ (f|_T)\in\widetilde{\mathcal{A}}_{[\M]}^{u}(T)$.
\end{definition}

\begin{remark}
In the previous definition, the condition $c_0^f=0$ guarantees, thanks to the expression~\eqref{eq.coeff.asymp.expan.lim.remainder} for $n=0$, that $0\in\overline{f(T)}$ for every subsector $T$ of $S$, and so it makes sense to consider functions $g$ in $\widetilde{\mathcal{A}}_{[\M]}^{u}(f(T))$. 
\end{remark}

The next auxiliary result studies the growth of the remainders for the powers $f^j$, $j\in\NN$, for a function $f\in\widetilde{\mathcal{A}}_{\M,h}^{u}(S)$ with $c_0^f=0$. We note first that, if $n\in\NN_0$ and $n\le j$, the fact that $c_0^f=0$ implies that $r_{f^j}(z,n)=f^j(z)$ for $z\in S$. Since $f(z)=r_f(z,0)=r_f(z,1)$, we know that $|f(z)|\le \min\{A,AhM_1|z|\}$ for every $z\in S$ and some $A>0$, and so
\begin{equation}\label{eq.bound.remainders.fpowerj.n_at_most_j}
    |r_{f^j}(z,n)|\le A^j(\min\{1,hM_1|z|\})^j,\quad z\in S.
\end{equation}
We treat now the case $n>j\ge 1$.

\begin{lemma}\label{lemma.bounds.powers}
Let $\M=(M_j)_j\in\RR_{>0}^{\NN_0}$ be a sequence with $M_0=1$, $S$ be a sector and $f\in\widetilde{\mathcal{A}}_{\M,h}^{u}(S)$ with $c^f_0=0$, so that there exists $A>0$ such that
\begin{equation}\label{eq.bounds.fpower1}
    |r_f(z,n)|\le A h^n M_n |z|^n,\quad z\in S,\ n\in\NN_0.
\end{equation}
Then, for every $j,n\in\NN$ with $n>j$ one has
\begin{equation}\label{eq.bounds.fpowerj}
    |r_{f^j}(z,n)|\le A^j h^n \left(\sum_{\substack{m_1+\dots+m_j=n\\m_1\ge 1,\dots,m_j\ge 1}}M_{m_1}\cdot\ldots \cdot M_{m_j}\right) |z|^n,\quad z\in S.
\end{equation}
\end{lemma}
\begin{remark}
Observe that the estimates~\eqref{eq.bounds.fpowerj} are also valid when $n=j$, since they take the form
$$
|r_{f^j}(z,j)|\le A^jh^jM_1^j|z|^j,\quad z\in S,
$$
which was already obtained in~\eqref{eq.bound.remainders.fpowerj.n_at_most_j}.
\end{remark}
\begin{proof}
We proceed by induction on $j$. For $j=1$ it is clear that the estimates in~\eqref{eq.bounds.fpowerj} reduce to those in~\eqref{eq.bounds.fpower1} particularized for $n>1$. Suppose now that the estimates~\eqref{eq.bounds.fpowerj} hold true for some $j-1\ge 1$ and every $n>j-1$, and we treat the $j$-th power. 
Observe that, whenever $1\leq i \leq n-j$, one has $n-i\geq n-(n-j)>j-1$, and so we can apply the induction hypothesis in order to estimate $r_{f^{j-1}}(z,n-i)$ in the formula~\eqref{eq.remaindern.fpowerj}. Moreover, we can use~\eqref{cjestimation} and write
\begin{align*}
|r_{f^j}(z,n)|&\leq \sum_{i=1}^{n-j}Ah^iM_i|z|^iA^{j-1}h^{n-i}\left(\sum_{\substack{m_1+\dots+m_{j-1}=n-i\\ m_1\ge 1,\dots,m_{j-1}\ge 1}}M_{m_1}\cdot\ldots \cdot M_{m_{j-1}}\right)|z|^{n-i}\\
&+Ah^{n-j+1}M_{n-j+1}|z|^{n-j+1}(AhM_1|z|)^{j-1}\\
&=A^j h^n \sum_{i=1}^{n-j+1}M_i\left(\sum_{\substack{m_1+\dots+m_{j-1}=n-i\\m_1\ge 1,\dots,m_{j-1}\ge 1}}M_{m_1}\cdot\ldots \cdot M_{m_{j-1}}\right) |z|^n\\
&=A^j h^n \left(\sum_{\substack{m_1+\dots+m_j=n\\m_1\ge 1,\dots,m_j\ge 1}}M_{m_1}\cdot\ldots \cdot M_{m_j}\right) |z|^n,
	\end{align*}
as desired.
\end{proof}

We are ready for the following generalization of \cite[Thm. 6]{AubersonMennessier81}.

\begin{proposition}\label{prop.stablecomposition}
Let $\M=(M_j)_j\in\RR_{>0}^{\NN_0}$ be a sequence with $M_0=1$, and let $S$ be a sector. If $\M$ satisfies \hyperlink{s-FdB}{$(\on{FdB})$}, then the class $\mathcal{\widetilde{A}}^{u}_{[\M]}(S)$ is closed under composition. 
\end{proposition}
\begin{proof}
Let $f\in\mathcal{\widetilde{A}}^{u}_{[\M]}(S)$ with $c_0^f=0$, $T$ be a subsector of $S$, and  $g\in\mathcal{\widetilde{A}}^{u}_{[\M]}(f(T))$. We have $f|_T\in\mathcal{\widetilde{A}}^{u}_{\M,h_1}(T)$ and $g\in\mathcal{\widetilde{A}}^{u}_{\M,h_2}(f(T))$ for some (resp. all) $h_i>0$, $i=1,2$, so that there exist $A_1=A_1(h_1)>0$ and $A_2=A_2(h_2)>0$ such that
\begin{equation*}
    |r_f(z,n)|\le A_1h_1^n M_n |z|^n,\quad z\in T,\ n\in\NN_0,
\end{equation*}
and
\begin{equation*}
    |r_g(f(z),n)|\le A_2 h_2^n M_n |f(z)|^n,\quad z\in T,\ n\in\NN_0.
\end{equation*}
Put $\varphi=g\circ f|_T$; according to the formula~\eqref{eq.equ75.roumieu}, and subsequently using Lemma~\ref{lemma.bounds.powers} and the fact that $f(z)=r_f(z,1)$, for every $n\in\NN_0$ and $z\in T$ we have
\begin{align*}
|r_{\varphi}(z,n)|&\le\sum_{j=1}^{n-1}|c^g_j||r_{f^j}(z,n)|+|r_g(f(z),n)|\\
&\le \sum_{j=1}^{n-1} A_2 h_2^j M_j \cdot A_1^j h_1^n 
\left(\sum_{\substack{m_1+\dots+m_j=n\\m_1\ge 1,\dots,m_j\ge 1}}M_{m_1}\cdot\ldots \cdot M_{m_j}\right) |z|^n
\\
&\quad + A_2 h_2^n M_n \left( A_1 h_1 M_1 |z| \right)^n\\
&=\sum_{j=1}^{n} A_2 h_2^j M_j \cdot A_1^j h_1^n 
\left(\sum_{\substack{m_1+\dots+m_j=n\\m_1\ge 1,\dots,m_j\ge 1}}M_{m_1}\cdot\ldots \cdot M_{m_j}\right) |z|^n.
\end{align*}
By assumption, $\M$ satisfies \hyperlink{s-FdB}{$(\on{FdB})$}, so there exist $C,B>0$ such that, for any $m_1,\dots,m_j\in\NN$ with $m_1+\dots+m_j=n$, one has 
$$
M_jM_{m_1}\cdot\ldots \cdot M_{m_j}\le CB^nM_n.
$$
Since the number of compositions of $n$ as the sum of $j$ positive integers is  $\binom{n-1}{j-1}$, we can write
\begin{align*}
|r_{\varphi}(z,n)|
&\le C A_2 (Bh_1)^n M_n |z|^n 
\sum_{j=1}^{n} \binom{n-1}{j-1} (A_1 h_2)^j\\
&\le C A_2 \left( Bh_1(1 + A_1 h_2) \right)^n M_n |z|^n .
\end{align*}
So, we deduce that $\varphi\in\mathcal{\widetilde{A}}^{u}_{\M,h_3}(T)$ for $h_3=Bh_1 (1 + A_1 h_2)$, and we are done in the Roumieu case. In the Beurling case, given an arbitrary $h_3>0$ one can take for $f$ the value $h_1>0$ such that $Bh_1(1+h_1)=h_3$, which will provide a value $A_1>0$ associated to $h_1$ for $f$; then, one can choose for $g$ the value $h_2=h_1/A_1>0$ and we conclude.
\end{proof}

Conversely and as before, whenever we have suitable characteristic functions in the Roumieu class, closure under composition implies condition \hyperlink{s-FdB}{$(\on{FdB})$}. 

\begin{theorem}\label{thm.stablecomposition}
Let $\M=(M_j)_j\in\RR_{>0}^{\NN_0}$ be a sequence with $M_0=1$ and $\alpha>0$. If $\alpha\leq 2$ we assume that $\overline{\G}^{2-\alpha}\M$ is equivalent to an \hyperlink{lc}{$(\text{lc})$} sequence, and if $\alpha>2$ we assume that there exists $\a'>\a$ such that $\overline{\G}^{2-\alpha'}\M$ is equivalent to an \hyperlink{lc}{$(\text{lc})$} sequence. Then the class $\mathcal{\widetilde{A}}^{u}_{\{\M\}} (S_\alpha)$ is closed under composition if and only if $\M$ satisfies  \hyperlink{s-FdB}{$(\on{FdB})$}. 
\end{theorem}
 \begin{proof} By Proposition~\ref{prop.stablecomposition}, we only need to show necessity. 
 
By Theorem~\ref{mainTthm1}, there exists a characteristic function $f\in \mathcal{\widetilde{A}}^{u}_{\{\M\}} (S_\alpha)$ in the strong sense, i.e., $f\sim^u_{\{\M\}} \sum_{j=0}^\infty c^f_j z^j$ and $(|c^f_j|)_{j\in\NN_0}$ is equivalent to $\M$. Moreover, by Remark~\ref{rem.signofcoefficients} we know that $c_0^f>0$ and $c_1^f<0$. Let us consider the function $f_0$ defined in $S_\a$ by $f_0(z)=c_0^f-f(z)$. It is clear that $f_0\in \mathcal{\widetilde{A}}^{u}_{\{\M\}} (S_\alpha)$ and $f_0\sim^u_{\{\M\}} -\sum_{j=1}^\infty c^f_j z^j$ (in particular, $c_0^{f_0}=0$). 

In case $\alpha\le 2$, choose real numbers $\b$ and $\varepsilon$ such that $0<\varepsilon<\frac{\pi}{2}(\alpha-\b)<\frac{\pi}{2}$. Consider $\varepsilon_1:=-c_1^f\sin(\varepsilon)<|c_1^f|$.
Due to~\eqref{eq.limit.derivatives.coeff}, we have that
$$
\lim_{\substack{z\to 0\\z\in S_{\b}}}f_0'(z)=-c_1^f,
$$
and so there exists $r>0$ such that for every $z\in S_{\b,r}=\{w\in S_{\b}\colon |w|<r\}$ one has $|f_0'(z)+c_1^f|<\varepsilon_1$. Take $z\in S_{\b,r}$, then the segment $(0,z]$ is contained in $S_{\b,r}$ and so,
$$
|f_0(z)+c_1^fz|=\left|\int_{[0,z]}(f_0'(w)+c_1^f)\,dw\right|\le \varepsilon_1|z|.
$$
For such $z$ we deduce that
$$
|\arg(f_0(z))-\arg(z)|=|\arg(f_0(z))-\arg(-c_1^fz)|\le \arcsin\left(\frac{\varepsilon_1|z|}{-c_1^f|z|}\right)=\varepsilon.
$$
The choices of $\b$ and $\varepsilon$ guarantee that $|\arg(f_0(z))|<\frac{\a\pi}{2}$, and so $f_0(S_{\b,r})\subset S_{\a}$. Since $\mathcal{\widetilde{A}}^{u}_{\{\M\}} (S_\alpha)$ is closed under composition, we know that $\varphi:= f\circ f_0 \in \mathcal{\widetilde{A}}^{u}_{\{\M\}} (S_{\b,r})$ (where, for simplicity, we have written $f_0$ instead of $f_0|_{S_{\b,r}}$, and $f$ instead of $f|_{f_0(S_{\b,r})}$).

In particular, if $\varphi\sim^u_{\{\M\}}\sum_{n=0}^\infty c_n^{\varphi}z^n$, we know that there exist $A_1>0$ and $h_1>0$ such that 
\begin{equation}\label{eq.estimates.cn.composition}
|c^{\varphi}_n|\le A_1h_1^nM_n, \quad n\in\NN_0.     
\end{equation}
At the same time, the formulas~\eqref{equ71} and~\eqref{equ74} allow us to write $c^\varphi_0=c^f_0$ and, since $c_0^{f_0}=0$ and using again Remark~\ref{rem.signofcoefficients},
\begin{align*}
c^\varphi_n&=\sum_{j=1}^{n}c^f_jc^{f_0^j}_n=
\sum_{j=1}^{n}c^f_j\sum_{m_1+\dots+m_j=n}c^{f_0}_{m_1}\cdots c^{f_0}_{m_j}\\
&=\sum_{j=1}^{n}(-1)^j|c^f_j|\sum_{\substack{m_1+\dots+m_j=n\\m_1\ge 1,\dots,m_j\ge 1}}(-(-1)^{m_1}|c^{f}_{m_1}|)\cdots (-(-1)^{m_j}|c^{f}_{m_j}|)\\
&=(-1)^n\sum_{j=1}^{n}|c^f_j|\sum_{\substack{m_1+\dots+m_j=n\\m_1\ge 1,\dots,m_j\ge 1}}|c^{f}_{m_1}|\cdots |c^{f}_{m_j}|,\;\;\;n\in\NN. 
\end{align*}
As $(|c^f_j|)_{j\in\NN_0}$ is equivalent to $\M$, there exist $A_2,B>0$ such that $|c^f_j|\geq A_2 B^j M_j$ for all $j\in\NN_0$. Then,
\begin{align*}
|c^\varphi_n|&=\sum_{j=1}^{n}|c^f_j|\sum_{\substack{m_1+\dots+m_j=n\\m_1\ge 1,\dots,m_j\ge 1}}|c^{f}_{m_1}|\cdots |c^{f}_{m_j}|\\
&\ge \sum_{j=1}^{n}A_2 B^j M_j\sum_{\substack{m_1+\dots+m_j=n\\m_1\ge 1,\dots,m_j\ge 1}}A_2 B^{m_1}M_{m_1}\cdots A_2 B^{m_j} M_{m_j}\\
&=A_2 B^n\sum_{j=1}^{n} (A_2B)^jM_j\sum_{\substack{m_1+\dots+m_j=n\\m_1\ge 1,\dots,m_j\ge 1}}M_{m_1}\cdots M_{m_j}.
\end{align*}
In case $A_2B\ge 1$, we deduce that, for every $m_1,\dots,m_j\in\NN$ with $\sum_{i=1}^jm_i=n$, we have $|c_n^{\varphi}|\ge A_2B^nM_jM_{m_1}\dots M_{m_j}$, and we obtain the condition \hyperlink{s-FdB}{$(\on{FdB})$} when comparing with the inequalities~\eqref{eq.estimates.cn.composition}. If $A_2B<1$, we get
$|c_n^{\varphi}|\ge A_2(A_2B^2)^nM_jM_{m_1}\dots M_{m_j}$, and the conclusion follows similarly.

Finally, the case $\a>2$ is treated by considering the restriction of a charasteristic function $f$ in the class $\mathcal{\widetilde{A}}^{u}_{\{\M\}} (S_\alpha)$ to the sector $S_2$, which can be identified with $\CC\setminus(-\infty,0]$, and proceeding as before to obtain a proper subsector $S_{\b,r}$, with $\b\in(1,2)$ and $r>0$, so that the image of $S_{\b,r}$ under the function $c_0^f-f$ is contained in $S_2$. This allows us to reason as before with the function $f\circ(c_0^f-f)$ (after suitable restrictions), which belongs to the corresponding class by hypothesis, and we can deduce the  condition \hyperlink{s-FdB}{$(\on{FdB})$} again.
\end{proof}

\vskip.5cm
\noindent\textbf{Acknowledgements}: The first three authors are partially supported by the Spanish Ministry of Science and Innovation under the project PID2022-139631NB-I00. The research of the fourth author was funded in whole by the Austrian Science Fund (FWF) project 10.55776/PAT9445424.\par

\bibliographystyle{plain}

\begin{thebibliography}{10}


%
%
	
	\bibitem{AubersonMennessier81}
	G.~Auberson and G.~Mennessier, Some properties of {B}orel summable functions, J. Math. Phys. 22 (1981), 2472-2481.


	\bibitem{ultraholomstab}
	J.~Jim{\'e}nez-Garrido, I.~Miguel-Cantero, J.~Sanz and G.~Schindl, Stability properties of ultraholomorphic classes of {R}oumieu-type defined by weight matrices, Rev. Real Acad. Cienc. Exactas Fis. Nat. Ser. A-Mat. RACSAM 118, 85 (2024).
	
	
	\bibitem{sectorialextensions}
	J.~Jim{\'e}nez-Garrido, J.~Sanz and G.~Schindl, Sectorial extensions, via Laplace transforms, in ultraholomorphic
	classes defined by weight functions, Results Math. 74 (2019), no. 1, 27.
	
	\bibitem{sectorialextensions1}
	J.~Jim{\'e}nez-Garrido, J.~Sanz and G.~Schindl, Sectorial extensions for ultraholomorphic classes defined by weight	 functions, Math. Nachr. 293 (2020), no. 11, 2140--2174.
	
	
	\bibitem{Komatsu73}
	H.~Komatsu, Ultradistributions, I. Structure theorems and a characterization, J. Fac. Sci. Univ. Tokyo Sect. IA Math. 20 (1973), 25--105.

%

	

	\bibitem{MozoPhd}
	J.~Mozo-Fernández, Teoremas de división y de Malgrange-Sibuya para funciones con desarrollo asintótico fuerte en varias variables, Ph.D. thesis, University of Valladolid, 1996.

    \bibitem{Poschel}
    J. Pöschel, On the Siegel-Sternberg linearization theorem. J. Dyn. Diff. Equat. 33, 1399--1425 (2021). 
	
	\bibitem{compositionpaper}
	A.~Rainer and G.~Schindl, Composition in ultradifferentiable classes, Studia Math. 224 (2014), no. 2, 97--131.
	
	\bibitem{characterizationstabilitypaper}
	A.~Rainer and G.~Schindl, Equivalence of stability properties for ultradifferentiable function classes, Rev. R. Acad. Cienc. Exactas Fís. Nat. Ser. A Mat. RACSAM 110 (2016), no. 1, 17--32.
%
%
%
	\bibitem{salinas62}
	B.~Rodríguez-Salinas, Clases de funciones analíticas, clases semianalíticas y cuasianalíticas, Rev. R. Acad. Cienc. Exactas Fís. Quím. Nat. Zaragoza (2) 17 (1962), 5--75.
%
%
%
	
	
	
%
%
%
%
	
	\bibitem{thilliez}
	V.~Thilliez, On quasianalytic local rings, Expo. Math. 26 (2008), no. 1, 1--23.
\end{thebibliography}

%
%
%

\vskip.5cm
\noindent\textbf{Affiliations}:\\
\noindent Javier~Jim\'{e}nez-Garrido:\\
Departamento de Matem\'aticas, Estad{\'\i}stica y Computaci\'on\\
Universidad de Cantabria\\
Avda. de los Castros, s/n, 39005 Santander, Spain\\
Instituto de Investigaci\'on en Matem\'aticas IMUVA, Universidad de Va\-lla\-do\-lid\\
ORCID: 0000-0003-3579-486X\\
E-mail: jesusjavier.jimenez@unican.es\\

\vskip.1cm
\noindent
Ignacio Miguel-Cantero:\\
Departamento de \'Algebra, An\'alisis Matem\'atico, Geometr{\'\i}a y Topolog{\'\i}a\\
Universidad de Va\-lla\-do\-lid\\
Facultad de Ciencias, Paseo de Bel\'en 7, 47011 Valladolid, Spain.\\
Instituto de Investigaci\'on en Matem\'aticas IMUVA\\
ORCID: 0000-0001-5270-0971\\
E-mail: ignacio.miguel@uva.es\\

\vskip.1cm
\noindent Javier~Sanz:\\
Departamento de \'Algebra, An\'alisis Matem\'atico, Geometr{\'\i}a y Topolog{\'\i}a\\
Universidad de Va\-lla\-do\-lid\\
Facultad de Ciencias, Paseo de Bel\'en 7, 47011 Valladolid, Spain.\\
Instituto de Investigaci\'on en Matem\'aticas IMUVA\\
ORCID: 0000-0001-7338-4971\\
E-mail: javier.sanz.gil@uva.es\\

\vskip.1cm
\noindent Gerhard~Schindl:\\
Fakult\"at f\"ur Mathematik, Universit\"at Wien,
Oskar-Morgenstern-Platz~1, A-1090 Wien, Austria.\\
ORCID: 0000-0003-2192-9110\\
E-mail: gerhard.schindl@univie.ac.at
\end{document}